\documentclass[11pt]{article}
\usepackage{pbox}
\usepackage[tmargin=1in,bmargin=1in,lmargin=1in,rmargin=1in]{geometry}

\usepackage{amsthm}
\usepackage{amsmath}
\usepackage{amssymb}
\usepackage{bm}
\usepackage{mathrsfs}
\usepackage[dvips]{graphicx}

\usepackage{algorithmic}
\usepackage{algorithm}

\usepackage{color}

\usepackage{url}

\usepackage{supertabular}

\newtheorem{thm}{Theorem}
\newtheorem{cor}[thm]{Corollary}
\newtheorem{lemma}[thm]{Lemma}

\theoremstyle{definition}

\theoremstyle{remark}

\numberwithin{thm}{section}

\DeclareMathAlphabet{\mathsfsl}{OT1}{cmss}{m}{sl}

\renewcommand{\phi}{\varphi}

\newcommand{\argmin}{\operatorname*{arg\; min}}
\newcommand{\argmax}{\operatorname*{arg\; max}}

\def\reals{\mathbb{R}}

\newcommand{\bepsilon}{\boldsymbol{\epsilon}}
\newcommand{\bx}{\boldsymbol{x}}

\newcommand{\ba}{\boldsymbol{a}}
\newcommand{\bb}{\boldsymbol{b}}
\newcommand{\by}{\boldsymbol{y}}

\newcommand{\bz}{\boldsymbol{z}}
\newcommand{\bX}{\boldsymbol{X}}

\newcommand{\bD}{\boldsymbol{D}}

\def\reals{\mathbb{R}}
\def\bx{\boldsymbol{x}}

\def\bu{\boldsymbol{u}}

\def\b0{\mathbf{0}}

\def\bv{\boldsymbol{v}}
\def\bd{\boldsymbol{d}}

\def\calI{\mathcal{I}}
\def\sign{\mathrm{sign}}
\def\bI{\mathbf{I}}

\usepackage{fancyhdr}

\vspace{-2ex}
\title{Element-wise estimation error of a total variation regularized estimator for change point detection}
\author{Teng Zhang
\thanks{The Department
of Mathematics, University of Central Florida, Orlando,
FL, 32765 USA e-mail: teng.zhang@ucf.edu.}
}
\date{\vspace{-5ex}}
\date{}  
\date{\today}

\begin{document}

\newcommand{\e}{\epsilon}
\newcommand{\CH}{{\mathcal H}}
\newcommand{\CB}{\mathcal B}
\newcommand{\Tin}{T_{\text{in}}}
\newcommand{\Tex}{T_{\text{ex}}}
\newcommand{\Iin}{I_1}
\newcommand{\Iex}{I_2}
\newcommand{\mbr}{{\mathbb R}}
\maketitle

\begin{abstract}
This work studies the total variation regularized $\ell_2$ estimator (fused lasso) in the setting of a change point detection problem. Compared with existing works that focus on the sum of squared estimation errors, we give bound on the element-wise estimation error. Our bound is nearly optimal in the sense that the sum of squared error matches the best existing result, up to a logarithmic factor. This analysis of the element-wise estimation error allows a screening method that can approximately detect all the change points. We also generalize this method to the muitivariate setting, i.e., to the problem of group fused lasso.
\end{abstract}
\vspace{-1ex}
\section{Introduction}
The problem of change point detection has applications in various fields including statistics, engineering, economics, and biostatistics, where the goal is to partition a signal into several homogeneous segments of variable durations, in which some quantity remains approximately constant over time. This issue was addressed in a large literature (see~\cite{bhattacharya} for a review), where the problem was tackled both from an online/sequential and an offline/retrospective points of view. In the online point of view, the data is received as a stream and the goal is to detect the change  in real time~\cite{moustakides2008}; in the offline point of view, the data is received as a data set and the goal is to find the change point based on the full data set. In this work, we will focus on the offline approach.

When there is a single change point, the detection of the change point is usually obtained by running a statistical test, for example, a log likelihood ratio test~\cite{csorgo1997limit}. The problem is more difficult when there are multiple change points, and there are methods based on exhaustive search with the Schwarz criterion~\cite{Yao1988,Yao1989}, or with binary segmentation method~\cite{doi:10.1093/biostatistics/kxh008,doi:10.1093/bioinformatics/btl646,fryzlewicz2014}, or optimization methods such as nonconvex approach \cite{doi:10.1111/rssb.12047} or the convex approach with $\ell_1$ penalization~\cite{doi:10.1093/bioinformatics/bti646,Tibshirani05sparsityand,Tibshirani2008}. The $\ell_1$ penalization method,  also referred as fused lasso or total variation regularized $\ell_2$ estimator in literature, is widely used for image denoising and piecewise constant signal estimation with respect to a given undirected graph~\cite{Rudin1992}. There are also methods based on two-step procedures, which first define a local diagnostic measure to  find a list of candidates that are most likely to be the change points, and then investigate these candidates, by methods such as hypothesis testing \cite{doi:10.1080/15326348808807089,Gijbels1999,Niu2012}. While there is also abundant literature that addresses the  change-point estimation problem from a Bayesian point of view; see \cite{ruanaidh1996numerical,Fearnhead2006}, but we will not consider this approach since it is very different in both modeling and algorithm.

The difficulty with the exhaustive search method lies in is its computational cost: the algorithm based on dynamic programming has a computational complexity of $O(n^2)$~\cite{10.2307/2673465}, where $n$ is the number of signals. Therefore, it is computationally expensive for large $n$. The binary segmentation method, which applies the single change-point test recursively to determine all the change-points, but it would hardly detect small segments buried in larger ones~\cite{doi:10.1093/biostatistics/kxh008}. Remedies in \cite{doi:10.1093/bioinformatics/btl646} again require algorithms with computational costs of $O(n^2)$. Optimization problems that involve nonconvex penalties or constraints are usually solved by dynamic programming, which again requires $O(n^2)$ \cite{doi:10.1111/rssb.12047}. In comparison, there are faster algorithms for fused lasso since it is convex, for example, there is an algorithm with a computational cost $O(n\log n)$ in \cite{Harchaoui2010}. In addition, many exisiting efficient algorithms for lasso can also be applied total variation regularized estimator (fused lasso) as a special case.

This work aims to develop new theories for change point detection problem with total variation regularized estimator (fused lasso). While there are many existing theoretical works on the fused lasso, most of them focus on the sum of squared estimation errors,  and they tend to assume that the regularization parameter is fixed as a prespecified value. The approach in this work is very different from the existing analysis: first, it investigates the element-wise estimation error for each of the observations, and second, the analysis holds for any regularization parameter. Since this method gives element-wise estimation errors, it can be used to construct a screening method to approximately detect all change points. In addition, this work investigates its generation to the setting of group fused lasso, which has not been discussed in existing works to the best of our knowledge. 

The manuscript is organized as follows. Background on fused lasso and existing theoretical guarantees are reviewed in Section 2. Then we proceed with our theoretical analysis of fused lasso in Section 3, and its generalization to group fused lasso in Section 4.

\section{Mathematical Background and Context of Proposed Research}\label{sec:background}
\subsection{Background}
The fused lasso/total variation regularized $\ell_2$ estimator estimates a signal from its noisy observation, based on the assumption that the signal tends to be piecewise constant. Specifically, we consider the following model:
\begin{equation}\label{eq:model}
y_i=x_i+\epsilon_i, 1\leq i\leq n,\,\,\,\, x_i, y_i,\epsilon_i\in\reals.
\end{equation}
Here $\{x_i\}_{i=1}^n$ is the underlying signal that exhibits the piecewise constant property such that the total number of distinct values in the set is given by $K$:  there exists $\{n_k\}_{k=1}^K$ such that $1=n_1<n_2<\cdots<n_K<n=n_{K+1}-1$ and for any $1\leq k\leq K$,
\[
x_{n_k}=x_{n_k+1}=\cdots=x_{n_{k+1}-1},
\]
$\{y_i\}_{i=1}^n$ is the noisy observation of the underlying signal, and $\epsilon_i$ is the noise that is i.i.d. from the normal distribution $N(0,\sigma^2)$. 
The two questions are, first, can we recover the signal $\{x_1,\cdots,x_n\}$ from its noisy observation $y_1,\cdots,y_n$? Second, can we detect the change points $\{n_2,\cdots,n_K\}$? The total variation regularized estimation/fused lasso approach proposes to answer these questions by solving the following optimization problem
\begin{equation}\label{eq:optimization}
\hat{x}_1,\cdots,\hat{x}_n=\argmin_{x_1,\cdots,x_n\in\reals}\sum_{i=1}^n(x_i-y_i)^2+\lambda\sum_{i=1}^{n-1}|x_i-x_{i+1}|.
\end{equation}
This problem is named ``fused lasso'' since it can be considered as a special case of the lasso problem. Let $\bD\in\reals^{n-1\times n}$ be defined such that $\bD_{i,i}=1$ and $\bD_{i,i+1}=-1$, and other elements are all zero, then \eqref{eq:optimization} becomes
\begin{equation}\label{eq:lasso}
\min \sum_{i=1}^n(x_i-y_i)^2 +\lambda \|\bD(x_1,\cdots,x_n)^T\|_1,
\end{equation} which can be considered as a generalized lasso problem \cite{Tibshirani2011,Zhu2015,Arnold2016}, and with a reformulation, it is equivalent to the lasso problem.

There are several ways of generalizing the model \eqref{eq:model} and its associated optimization problem  \eqref{eq:optimization}. In particular, we will investigate
%
%
group fused lasso \cite{Alaiz:2013:GFL:2731664.2731674,WytSraKol2014,bleakley:hal-00602121} in this work. For this problem, instead of the assumption that $x_i, y_i,\epsilon_i\in\reals$ as in \eqref{eq:model}, we assume that they are vectors, then the model \eqref{eq:model} becomes
\begin{equation}\label{eq:gfl_model}
\by_i=\bx_i+\bepsilon_i, 1\leq i\leq n,\,\,\,\,\by_i,\bx_i,\bepsilon_i\in\reals^p,\,\,\,\bepsilon_i\sim N(\b0,\sigma^2\bI_{p\times p})
\end{equation}
and the optimization problem \eqref{eq:optimization} becomes
\begin{equation}\label{eq:gfl}
\hat{\bx}_1,\cdots,\hat{\bx}_n=\argmin_{\bx_1\cdots,\bx_n\in\reals^p}\sum_{i=1}^n\|\by_i-\bx_i\|^2+\lambda\sum_{i=1}^{n-1}\|\bx_i-\bx_{i+1}\|.
\end{equation}
This model has applications when the observations are vectors instead of scalars.

\subsection{Existing theoretical results}\label{sec:currentheory}
We begin by reviewing some existing results on the property of the solution to \eqref{eq:optimization}. In fact, many works have been done for bounding the sum of squared estimation error $\sum_{i=1}^n(\hat{x}_i-x_i)^2$. For example, \cite{mammen1997} showed that when $\sum_{i=1}^{n-1}\|x_i-x_{i+1}\|\leq C_n$  for a nondecreasing sequence $C_n$, then for the choice that  $\lambda=O(n^{1/3}C_n^{-1/3})$, the estimator from \eqref{eq:optimization} satisfies
\[
\frac{1}{n}\sum_{i=1}^n(\hat{x}_i-x_i)^2=O(n^{-2/3}C_n^{2/3}).
\]

Dalalyan et al. \cite{dalalyan2017} showed that when $\lambda=\sigma\sqrt{2n\log (n/\delta)}$, the estimation error is bounded above by
\[
\frac{1}{n}\sum_{i=1}^n(\hat{x}_i-x_i)^2\leq c\sigma^2\frac{K\log(n/\delta)}{n}\left(\log n+\frac{n}{W_n}\right),
\]
where probability $1-2\delta$, for all $\delta>0$ and $n\geq N$, where $c,N>0$ are constants and $W_n$ is the minimal distance between the jumps, i.e., $W_n=\min_{1\leq k\leq K}(n_{k+1}-n_{k})$.

Harchaoui and L\'{e}vy-Leduc \cite[Proposition 2]{Harchaoui2010} proves that the fused lasso has estimation error $\frac{1}{n}\sum_{i=1}^n(\hat{x}_i-x_i)^2=O(\log n/n)$ when $\lambda$ is chosen to be in the order of $\sqrt{\log n/n}$, but the theory assumes that the number of change points in the estimator $\hat{\bx}$ is bounded, which is not verified and might be impractical.

Lin et al. \cite{Lin2016} showed that when $\lambda=(nW_n)^{1/4}$, the estimation error has the property of
\begin{equation}\label{eq:lin2016}
\frac{1}{n}\sum_{i=1}^n(\hat{x}_i-x_i)^2\leq \gamma^2 c\sigma^2\frac{K}{n}\left((\log K+\log\log n)\log n+\sqrt{\frac{n}{W_n}}\right)
\end{equation}
with probability as least $1-\exp(C\gamma)$, for some $c,C$ depending on $\sigma$.

Guntuboyina et al. \cite{Guntuboyina2017} analyzed the fused lasso problem along with trend filtering. For the fused lasso problem \eqref{eq:optimization} it shows that if for all $1\leq k\leq K$ such that $x_{n_{k-1}}-x_{n_{k-1}-1}$ and $x_{n_{k}}-x_{n_{k}-1}$ have different signs,  $n_k-n_{k-1}>cn/(K+1)$, then the optimal $\lambda$ produces $\hat{\bx}$ such that
\begin{equation}\label{eq:Guntuboyina2017}
\frac{1}{n}\sum_{i=1}^n(\hat{x}_i-x_i)^2\leq C\sigma^2\frac{K+1}{n}\log\left(\frac{en}{K+1}\right)+\frac{4\sigma^2\delta}{n}
\end{equation}
with probability at least $1-\exp(-\delta)$. However, while the optimal choice of $\lambda$ is given in the paper, it is rather complicated and difficult for practical use.

Ortelli and van de Geer \cite{Ortelli2018} recently improved the bounds in Dalalyan et al. \cite{dalalyan2017}:  when  $\lambda=O(\sigma\sqrt{2n\log (n/\delta)})$, then with probability $1-\delta$, the estimation error is bounded above by
\begin{equation}\label{eq:Ortelli2018}
\frac{1}{n}\sum_{i=1}^n(\hat{x}_i-x_i)^2\leq O\left(\sigma^2\frac{K\log(n/\delta)}{n}\left(\log \frac{n}{K}+\frac{n}{m_H}\right)\right),
\end{equation}
where the $m_H$ is the harmonic mean between the distances of jumps, i.e., the harmonic mean of $n_{k+1}-n_k$ for all $1\leq k\leq K$.

Besides controlling the sum of squared estimation errors $\frac{1}{n}\sum_{i=1}^n(\hat{x}_i-x_i)^2$, another important goal is to detect the change points $n_1,\cdots,n_K$. There have been some works that discuss how to apply fused lasso for this task as well (though not as many as the works on estimation error). Mathematically, we would like to estimate a set of change points from $\hat{\bx}$, denoted by $\hat{S}\in\{1,\cdots,n\}$, and find its distance with the set of true change points $S=\{n_1,\cdots,n_K\}$. A measure of the distance  between two sets $A$ and $B$ can be defined as follows:
\[
d_H(A,B)=\max(d(A,B),d(B,A)), \text{where}\,\,d(A,B)=\max_{b\in B}\min_{a\in A}|a-b|,
\]
and then the goal is to construct a set $\hat{S}$ such that $d_H(\hat{S},S)$ is minimized.

While intuitively one can use the change points of $\hat{x}_1, \cdots, \hat{x}_n$ to estimate $S$, i.e., let $\hat{S}=\{1\leq i]\leq n-1: \hat{x}_i\neq \hat{x}_{i+1}\}$, the method turns out to be impractical. Rojas and Wahlberg \cite{Rojas2014} established an impossibility result for the fused lasso estimator when $x_1,\cdots, x_n$ exhibits a ``staircase'' pattern, which means that $\{\hat{x}_{n_{k+1}}-\hat{x}_{n_{k}}\}_{k=1,\cdots,K}$ has two consecutive positive or negative values; specifically, these authors proved that under this setting, $d_H(\hat{S},S)/n$ remains bounded away from zero with nonzero asymptotic probability. That being said, for non-staircase patterns in $\bx$, the authors also showed, under certain assumptions, that $d_H(\hat{S},S)/n$ converges to zero in probability.

Qian and Jia \cite{Qian2016} studied a modification
of the fused lasso defined by transforming the fused lasso problem \eqref{eq:optimization} into a lasso problem with
particular design matrix $\bX$, and then applying a step that preconditions $\bX$. The authors concluded that exact recovery is possible with probability tending to 1, as long as the minimum signal gap $H_n$ and tuning parameter satisfy $H_n\geq \lambda=O(\log n)$. But this is a very strong requirement on the scaling of the signal gap $H_n$, as Sharpnack et al. \cite{pmlr-v22-sharpnack12} showed, even simple pairwise thresholding achieves exact recovery as well.

Harchaoui and L\'{e}vy-Leduc \cite[Proposition 3]{Harchaoui2010}  analyze the detection of changed point by $\hat{\bx}$ and showed that if the change points in $\hat{\bx}$ is exactly $K$, and when $\lambda$ is appropriately chosen, then every change point in $\bx$ would be detected approximately  in the sense that there exists a change point in $\hat{\bx}$, which has a location that is different by at most $n\delta_n$, where $\delta_n$ is chosen such that $n\delta_n\leq W_n$, $n\delta_n H_n^2/\log n\rightarrow\infty$, $n\delta_n W_n/\lambda\rightarrow\infty$, and $n\delta_n W_n^2/\log (n^3/\lambda^2)\rightarrow\infty$, and here $W_n$ is the smallest distance between change points: $W_n=\min(n_2-n_1,n_3-n_2,\cdots,n-n_K)$ and $H_n$ is the smallest distances between ``levels'': $H_n=\min_{k=1,\cdots,K-1}|x_{n_k}-x_{n_{k+1}}|$. However, it is unclear whether the assumption ``the change points in $\hat{\bx}$ is exactly $K$'' is reasonable in their setting.

Lin et al. \cite{Lin2016} construct a simple filtering-based technique to a general estimator $\{\hat{x}_i\}_{i=1}^n$ and obtain a set $\hat{S}_F(\hat{\bx})$ such as its distance to the change points $S=[n_1,\cdots,n_K]$ is bounded by $nR_nv_n/W_n^2$ almost surely as $n\rightarrow\infty$, where $\frac{1}{n}\sum_{i=1}^n(\hat{x}_i-x_i)^2=O(R_n)$, and $v_n$ is an arbitrary diverging sequence: $\Pr\left(d_H(\hat{S}_F(\hat{\bx}),S)\leq nR_nv_n/H_n^2\right)\rightarrow 1.$ For the fused lasso estimator, two results are as follows: 1. If $C_n=\sum_{k=1,\cdots,K-1}|x_{n_k}-x_{n_{k+1}}|$, $W_n=\Theta(n)$, $H_n=w(n^{1/6}C_n^{1/3}/\sqrt{W_n})$, then $\lambda=\Theta(n^{1/3}C_n^{-1/3})$ leads to
\[
\Pr\left(d_H(\hat{S}_F(\hat{\bx}),S)\leq O\left(\frac{n^{1/3}C_n^{2/3}}{H_n^2}\right)\right)\rightarrow 1.\]
2. If $K=O(1)$, $W_n=\Theta(n)$, $H_n=w(\sqrt{\log n(\log\log n)/n})$, then $\lambda=\Theta(\sqrt{n})$ leads to
\[
\Pr\left(d_H(\hat{S}_F(\hat{\bx}),S)\leq O\left(\frac{\log n \log\log n}{H_n^2}\right)\right)\rightarrow 1.\]

There are also works that discuss the properties of some closely related problems. For example, \cite{Qian2013} studies the structural change problems (which can be considered as the fused lasso problem \eqref{eq:optimization} in a regression setting) and studies its asymptotic properties. It showed that if $\lambda$  satisfies certain assumptions, then the group fused Lasso procedure cannot under-estimate the number of change points, and all change points can be consistently detected in the sense that for every true change point, there is an estimated change point close by. It also proposed a BIC-based method to choose $\lambda$ such that the estimated number of  change points is correct. However, it is unclear that if this chosen $\lambda$ still satisfies the assumption on $\lambda$ such that all underlying change points are detected.

\section{Fused lasso}\label{sec:fused}
As reviewed in Section~\ref{sec:currentheory}, there are two main goals in analyzing the solution to \eqref{eq:optimization}. First, find an upper bound for the sum of squared estimation errors $\sum_{i=1}^n(\hat{x}_i-x_i)^2$. Second,  finds the set of changes points from the estimation $\{\hat{x}_i\}_{i=1}^n$. In this section. we will first present the main result in Theorem~\ref{thm:fused} that analyzes the sum of squared estimation error $\hat{x}_i-x_i$ for all $1\leq i\leq n$. Based on this result, we will investigate the sum of squared estimation errors in Corollary~\ref{cor:estimationerror} and the detection of change points in Corollary~\ref{cor:detection}. The proof of Theorem~\ref{thm:fused} is given in Section~\ref{sec:proof_fused}, and the main ingredient of the proof, Theorem~\ref{thm:main}, is proved in Section~\ref{sec:proof_main}.

\subsection{Main results of the fused lasso}\label{sec:fused2}
\begin{thm}\label{thm:fused}
For all $1\leq k\leq K$,  let $m_k=n_{k+1}-n_{k}$ be the length of the $k$-th interval. For $1\leq i \leq n$, assume that the $i$-th observation lies in the $k(i)$-th interval, that is, $n_{k(i)}\leq i\leq n_{k(i)+1}-1$, and
\[
d_i=\min(i+1-n_{k(i)}, n_{k(i)+1}-i),
\]
that is, $d_i$ is its distance to the nearest change point.  Let $M_y=2\sigma\sqrt{\log n+\log t}$, then with probability $1-1/t^2$, for all $1\leq i\leq n$ we have
\begin{equation}\label{eq:main01}
|\hat{x}_i-x_i|\leq \max\left(\frac{M_y}{\sqrt{d_i}},\frac{M_y^2}{4\lambda},\frac{2\lambda}{m_{k(i)}}+\frac{2M_y}{\sqrt{m_{k(i)}}}\right).
\end{equation}
\end{thm}
We remark that this result is stronger than existing results in the literature \cite{Harchaoui2010,dalalyan2017,Lin2016,Guntuboyina2017,Ortelli2018} in the sense that it controls the element-wise estimation error for each of the signals, instead of controlling the sum of squared estimation errors. In addition, our estimation holds for any choice of $\lambda$. As a result, it gives a more detailed description of the solution of the fused lasso.
Applying Theorem~\ref{thm:fused}, with some calculation we obtain the following result on the sum of squared estimation errors of  \eqref{eq:optimization}:
\begin{cor}\label{cor:estimationerror}
With probability $1-1/t^2$, the following inequality holds for $M_y$ defined in Theorem~\ref{thm:fused}
\[
\frac{1}{n}\sum_{i=1}^n|\hat{x}_i-x_i|^2\leq \frac{M_y^4}{16\lambda^2}+\frac{8\lambda^2}{n}\sum_{k=1}^K\frac{1}{m_k}+\frac{2}{n}M_y^2\left(4+K+\sum_{k=1}^K\log m_k\right).
\]
\end{cor}
\begin{proof}
For the $k$-th interval, that is, for $i\in[n_k,n_{k+1})$, we have
\begin{align}\label{eq:est_main}
\sum_{i=n_k}^{n_{k+1}-1}|\hat{x}_i-x_i|^2\leq& \sum_{i=1}^{m_k}\left(\frac{M_y}{\sqrt{i}}\right)^2+\left(\frac{M_y}{\sqrt{m_k+1-i}}\right)^2+\left(\frac{M_y^2}{4\lambda}\right)^2+\left(\frac{2\lambda}{m_k}+\frac{2M_y}{\sqrt{m_k}}\right)^2
\\
\leq& m_k\left(\frac{M_y^2}{4\lambda}\right)^2+8\frac{\lambda^2}{m_k}+8M_y^2+2M_y^2(\log m_k+1),\nonumber
\end{align}
where the second inequality uses the well-known result of $1+\frac{1}{2}+\cdots+\frac{1}{m}\leq \log m+1$ and $(a+b)^2\leq 2a^2+2b^2$.

Combining the estimation in \eqref{eq:est_main} for $k=1,\cdots,K$, Corollary~\ref{cor:estimationerror} is proved.\end{proof}
We may compare Corollary~\ref{cor:estimationerror} with previous results in \cite{dalalyan2017,Lin2016,Guntuboyina2017,Ortelli2018}. For example,  \cite{Lin2016} shows that when $\lambda=(nW_n)^{\frac{1}{4}}$, where $W_n=\min_{i=1,\cdots,K}m_i$, then  \begin{equation}\label{eq:lin2016}
\frac{1}{n}\|\hat{\bx}-\bx\|^2\leq \gamma^2 c\sigma^2\frac{K}{n}\left((\log K+\log\log n)\log n+\sqrt{\frac{n}{W_n}}\right),\,\,\text{w.p. $1-\exp(-C\gamma)$}
\end{equation}
for $c, C$ depending on $\sigma$. If we plug in $\lambda=(nW_n)^{\frac{1}{4}}M_y$ to Corollary~\ref{cor:estimationerror}, we have
\begin{align*}\frac{1}{n}\|\hat{\bx}-\bx\|^2\leq &\frac{M_y^4}{16\sqrt{nW_n}}+\frac{8\sqrt{W_n}}{\sqrt{n}}\sum_{i=1}^K\frac{1}{m_i}+\frac{2}{n}M_y^2\left(4+K+\sum_{i=1}^K\log m_i\right)
\\\leq
&\frac{K}{n}M_y^2\left(2\left(5+\log \frac{n}{K}\right)+\sqrt{\frac{n}{W_n}}\left(\frac{8}{K}\sum_{i=1}^K\frac{W_n}{m_i}+\frac{1}{16K}\right)\right).
\end{align*}
 First, our estimation bound gives a more precise description of the dependence on $\sigma$. Second, considering that $M_y=O(\sqrt{\log n})$ and $\frac{8}{K}\sum_{i=1}^K\frac{W_n}{m_i}\leq 8$,  these two estimation bounds are in the same order up to a factor of $\log n$: the coefficient of $\sqrt{\frac{n}{W_n}}$ in \eqref{eq:lin2016}, $1$, is replaced with $M_y^2(\frac{8}{K}\sum_{i=1}^K\frac{W_n}{m_i}+\frac{1}{16K})$, and the term $(\log K+\log\log n)\log n$ in \eqref{eq:lin2016} is replaced with $M_y^2\log\frac{n}{K}$.

In the works \cite{dalalyan2017,Ortelli2018}, $\lambda$ is set to be $\sigma\sqrt{2n\log (n/\delta)}$, and then their results hold with probability $1-2\delta$. Here we use $\lambda=\sigma M_y\sqrt{n}$ (which is also in the order of $O(\sigma\sqrt{n\log n})$), then Corollary~\ref{cor:estimationerror}  implies that
\begin{equation}\label{eq:ouroptimal}
\frac{1}{n}\sum_{i=1}^n|\hat{x}_i-x_i|^2\leq O\left(\sigma^2 K\log n \frac{1}{m_H}+\sigma^2\frac{\log n}{n}\sum_{i=1}^K\log m_i\right),
\end{equation}
where $m_H$ is the harmonic mean of $m_1,\cdots,m_K$. Compared with the bound in \cite{Ortelli2018} (which improves over \cite{dalalyan2017}), which is in the order of $O(\sigma^2K\log n/m_H+\sigma^2K \log n  \log (n/K)/n )$, \eqref{eq:ouroptimal} is almost exactly the same and has a minor improvement from $K\log \frac{n}{K}$ to $\sum_{i=1}^K\log m_i$ ($\sum_{i=1}^K\log m_i\leq K\log \frac{n}{K}$ can be derived from the concavity of logarithmic function).



Compared with the result \eqref{eq:Guntuboyina2017} from \cite{Guntuboyina2017}, our estimation error at \eqref{eq:ouroptimal} is larger by a factor of $\log n$. However, we remark that the result from \cite{Guntuboyina2017} only holds for the optimal $\lambda$ and its formula  is rather complicated and difficult for practical use, and the analysis in \cite{Guntuboyina2017} also makes some assumptions on $m_k$, the length of intervals.

Theorem~\ref{thm:fused} allows us to approximately detect all change points through a simple screening of the solution to the fused lasso as follows.

\begin{cor}\label{cor:detection}
Assume $H_n$ is the smallest distances between the ``levels'' defined by $H_n=\min_{k=1,\cdots,K-1}|x_{n_k}-x_{n_{k+1}}|$ and $H_n \sqrt{W_n} >16 M_y$, then we can define $C>2$ such that $H_n \sqrt{W_n}  = 8 C M_y$. If $\lambda= (C-1) M_y\sqrt{W_n}$, then the set
\[
\hat{S}=\left\{i: |\hat{x}_{i-W_n/4C^2}-\hat{x}_{i+W_n/4C^2}|>\frac{H_n}{2}\right\},
\]
estimates the set of change points $S=\{n_2,\cdots,n_K\}$ approximately in the sense that with probability at least $1-1/t^2$,\[
d_H(\hat{S},S)\leq \frac{W_n}{2C^2} =  \frac{32M_y^2}{H_n^2}.
\]
\end{cor}
\begin{proof}
For $ W_n/4C^2\leq i \leq m- W_n/4C^2$, the upper bound in \eqref{eq:bound} is bounded above by $\frac{2CM_y}{\sqrt{W_n}}=H_n/4$. It then follows that $n_i\in \hat{S}$ and $[n_i+W_n/2C^2,n_{i+1}-W_n/2C^2]\not\in \hat{S}$, and Corollary~\ref{cor:detection} is proved.
\end{proof}
Compared with the theoretical guarantee for a similar approximate change point screening method in \cite{Lin2016}, which states that $d_H(\hat{S},S)\leq \log^2n/H_n^2$ when $K=O(1)$ and $W_n=\Theta(n)$, our result removes a factor of $\log n$ and does not make assumptions on $K$ and $W_n$. In fact, Corollary~\ref{cor:detection} matches the bound of  jump-penalized least squares estimators in \cite{10.2307/25464745}, binary segmentation (BS) and wild binary segmentation (WBS) in \cite{fryzlewicz2014}, the simultaneous multiscale change point estimator in \cite{doi:10.1111/rssb.12047}, and the tail-greedy unbiased Haar wavelets (TGUH)~\cite{fryzlewicz2018}.

\subsection{Proof of Theorem~\ref{thm:fused}}\label{sec:proof_fused}
We first present the main intermediate result that will be used to prove Theorem~\ref{thm:fused}:
\begin{thm}\label{thm:main}
Given $y_1,\cdots,y_m$ and $a,b\in\reals$, consider the problem:
\begin{equation}\label{eq:optimization1}
\{\hat{x}_i\}_{i=1}^m=\argmin_{x_1,\cdots,x_m\in\reals }G(x_1,\cdots,x_m),
\end{equation}
where
\[
 G(x_1,\cdots,x_m)=\sum_{i=1}^m(x_i-y_i)^2+\lambda\left(|x_1-a|+|x_m-b|+ \sum_{i=1}^{m-1}|x_i-x_{i+1}|\right).
\]
If for all $1\leq k\leq l\leq m$, we have
\begin{equation}\label{eq:main_assumption0}
\left|\sum_{i=k}^l y_i\right|\leq M_y \sqrt{l-k+1},
\end{equation}
then
\begin{equation}\label{eq:bound}
|\hat{x}_i|\leq \max\left(\frac{M_y}{\sqrt{i}},\frac{M_y}{\sqrt{m+1-i}},\frac{M_y^2}{4\lambda},\frac{2\lambda}{m}+\frac{2M_y}{\sqrt{m}}\right).
\end{equation}

\end{thm}
The following lemma establishes a property that happens with high probability and we will prove Theorem~\ref{thm:fused} based on the assumption that this property holds.
\begin{lemma}\label{lemma:prob}
With probability $1-1/t^2$, for $M_y=2\sigma\sqrt{\log n+\log t}$ and all $1\leq k\leq l\leq n$, we have
\begin{equation}\label{eq:prob_requirement}
\left|\sum_{i=k}^l\epsilon_i\right|\leq M_y \sqrt{l-k+1}.
\end{equation}
\end{lemma}
\begin{proof}[Proof of Lemma~\ref{lemma:prob}]
Use the common tail bound of the Gaussian distribution that $\Pr(N(0,1)>t)<\frac{1}{2}e^{-t^2/2}$ and $\sum_{i=k}^l \epsilon_i\sim N(0,(l-k+1)\sigma^2)$, we have that for any fixed $k,l$,
\[
\Pr\left(\left|\sum_{i=k}^l \epsilon_i\right|\leq M_y \sqrt{l-k+1}\right)\geq 1-\exp(-M_y^2/2).
\]
For all pairs of $1\leq k\leq l\leq n$, a union bound show that it holds with probability at least $1-n^2\exp(-M_y^2/2)=1-\exp(2\log n-M_y^2/2)=1-1/t^2$.
\end{proof}
\begin{proof}[Proof of Theorem~\ref{thm:fused}]
By the structure of the optimization problem \eqref{eq:optimization}, we have the following property
\begin{equation}\label{eq:optimization0}
\{\hat{x}_i\}_{i=n_k}^{i=n_{k+1}-1}=\argmin_{\{x_i\}_{i=n_k}^{n_{k+1}-1}} \lambda(|x_{n_k}-\hat{x}_{n_k-1}|+|x_{n_{k+1}-1}-\hat{x}_{n_{k+1}}|)+\sum_{i=n_k}^{n_{k+1}-1} \left(|y_i-x_i|^2+\lambda|x_i-x_{i+1}|\right).
\end{equation}
That is, to find $\hat{\bx}$ for any interval in $[n_k,n_{k+1}-1]$, we can investigate this subproblem that focuses on this interval, assuming that the correct estimations at the endpoints have been given. Because of this property, we proposed to study the optimization problem for each interval of constant values, i.e., $n_k\leq i\leq n_{k+1}-1$ with other estimations fixed. WLOG we may assume that $x_{n_k}=\cdots=x_{n_{k+1}-1}=0$ and $m=n_{k+1}-n_k$,  then the estimation problem \eqref{eq:optimization0} is reduced to the problem in the form of  \eqref{eq:optimization1}.

Applying Lemma~\ref{lemma:prob}, we know that the assumption \eqref{eq:main_assumption0} holds when we use Theorem~\ref{thm:main} to investigate \eqref{eq:optimization0} for  $k=1,\cdots,K$. Then the estimation error $|\hat{x}_i-x_i|$ of fused lasso is given by \eqref{eq:bound}, which then proves Theorem~\ref{thm:fused}.

We remark that for $k=1$ (and similarly for $k=K$), $\hat{x}_{n_1-1}$ does not exist and \eqref{eq:optimization0} would be reduced to
\[
\{\hat{x}_i\}_{i=n_1}^{i=n_{2}-1}=\argmin_{\{x_i\}_{i=n_1}^{n_{2}-1}} \lambda |x_{n_{2}-1}-\hat{x}_{n_{2}}|+\sum_{i=n_1}^{n_{2}-1} \left(|y_i-x_i|^2+\lambda|x_i-x_{i+1}|\right).
\]
While this is slightly different from the problem in Theorem~\ref{thm:main} (with one term in $G$ missing), using the same argument we would still reach the same bound \eqref{eq:bound} in Theorem~\ref{thm:main}.
\end{proof}
\subsection{Proof of Theorem~\ref{thm:main}}\label{sec:proof_main}
For the simplicity of notations in indices, we denote $a$ and $b$ by $\hat{x}_{0}$ and $\hat{x}_{m+1}$.

\begin{lemma}\label{lemma:localmax}
The following  two situations would not happen for any $1\leq m_1\leq m_2\leq m$:
\begin{align}\label{eq:localmax1}
\hat{x}_{m_1}=\hat{x}_{m_1+1}=\cdots=\hat{x}_{m_2}>\frac{M_y^2}{4\lambda},\,\,\,\hat{x}_{m_2}>\hat{x}_{m_2+1},\,\,\,\hat{x}_{m_1}>\hat{x}_{m_1-1}\\\label{eq:localmax2}
\hat{x}_{m_1}=\hat{x}_{m_1+1}=\cdots=\hat{x}_{m_2}<-\frac{M_y^2}{4\lambda},\,\,\,\hat{x}_{m_2}<\hat{x}_{m_2+1},\,\,\,\hat{x}_{m_1}<\hat{x}_{m_1-1}
\end{align}
Intuitively, this result implies that the sequence $\hat{x}_0,\hat{x}_1,\cdots,\hat{x}_m,\hat{x}_{m+1}$ does not have a local minimum smaller than $\frac{M_y^2}{4\lambda}$, or a local maximum larger than $\frac{M_y^2}{4\lambda}$.
\end{lemma}
\begin{proof}
We start with the situation as described in \eqref{eq:localmax1}. The proof is based on the construction of $\tilde{x}_i, i=1,\cdots,m$ as follows: $\tilde{x}_i=\hat{x}_i-\epsilon$ for $i\in [m_1,m_2]$, and $\tilde{x}_i=\hat{x}_i$ for $i\not\in[m_1,m_2]$. Then we have that
\[
\sum_{i=0}^{m}|\hat{x}_i-\hat{x}_{i+1}|-\sum_{i=0}^{m}|\tilde{x}_i-\tilde{x}_{i+1}|=2\epsilon
\]
and
\[
\sum_{i=1}^m(\hat{x}_i-y_i)^2-\sum_{i=1}^m(\tilde{x}_i-y_i)^2=2\epsilon\sum_{i=m_1}^{m_2}(\hat{x}_{i}-y_i)+O(\epsilon^2)
\]
As a result, $G(\{\hat{x}_i\}_{i=1}^m)\leq G(\{\tilde{x}_i\}_{i=1}^m)$ implies that
\[
0\geq 2\lambda+2\sum_{i=m_1}^{m_2}(\hat{x}_{i}-y_i)=2\lambda+2\sum_{i=m_1}^{m_2}(\hat{x}_{m_1}-y_i)
\]
and
\[
\hat{x}_{m_1}\leq \frac{\sum_{i=m_1}^{m_2}y_i-\lambda}{m_2-m_1+1}\leq \frac{M_y\sqrt{m_2-m_1+1}-\lambda}{m_2-m_1+1}\leq \frac{M_y^2}{4\lambda},
\]
where the last step follows from $\argmax_{x}a/x-\lambda/x^2=a^2/4\lambda$.

The proof for the situation described in \eqref{eq:localmax2} is identical (with different signs).
\end{proof}

Let $\mathcal{I}=\{0\leq i\leq m+1: |\hat{x}_i|\leq \frac{M_y^2}{4\lambda}\}$, then by Lemma~\ref{lemma:localmax}, $\mathcal{I}$ must be an interval in the form of $\mathcal{I}=\{m_1,m_1+1,\cdots,m_2\}$ or it is an empty set.

Case 1: $\mathcal{I}$ is nonempty.

Then $\mathcal{I}=\{m_1,m_1+1,\cdots,m_2\}$, and for all $i< m_1$ and $i>m_2$, $|\hat{x}|_i>\frac{M_y^2}{4\lambda}$. In addition, Lemma~\ref{lemma:localmax} implies that $\hat{x}_0,\cdots,\hat{x}_{m_1-1}$ have the same signs, that is, they are either all positive or all negative. Similarly, $\hat{x}_{m_2+1},\cdots,\hat{x}_{m+1}$ have the same signs.

Now let us investigate the sequence $\hat{x}_0,\cdots,\hat{x}_{m_1-1}$. WLOG we may assume that they are all positive. For any $k\leq m_1-1$, we define a sequence $\tilde{x}_i$ by
\[
\tilde{x}_i=\begin{cases}\hat{x}_i-\epsilon, \text{if $i\leq k'$}\\\hat{x}_i, \text{if $i>k'$},\end{cases}
\]
where $k'$ is the smallest integer such that
\[
\hat{x}_k=\hat{x}_{k+1}=\cdots=\hat{x}_{k'}\neq \hat{x}_{k'+1}.
\]
For example, if $\hat{x}_k\neq \hat{x}_{k+1}$, then $k'=k$. Since $\hat{x}_{m_1-1}\neq \hat{x}_{m_1}$ by definition, we have $k'\leq m_1-1$. Then
\[
\sum_{i=0}^{m}|\hat{x}_i-\hat{x}_{i+1}|-\sum_{i=0}^{m}|\tilde{x}_i-\tilde{x}_{i+1}|=0,
\]
and
\[
\sum_{i=1}^m(\hat{x}_i-y_i)^2-\sum_{i=1}^m(\tilde{x}_i-y_i)^2=2\epsilon\sum_{i=1}^{k'}(\hat{x}_{i}-y_i)+O(\epsilon^2).
\]
Therefore, $G(\{\hat{x}_i\}_{i=1}^m)\leq G(\{\tilde{x}_i\}_{i=1}^m)$ implies that
\[
0\geq \sum_{i=1}^{k'}(\hat{x}_{i}-y_i)\geq k'\hat{x}_k-\sum_{i=1}^{k'}y_i\geq k'\hat{x}_k-M_y\sqrt{k'},
\]
and then we have
\[
\hat{x}_k\leq \frac{M_y}{\sqrt{k'}}\leq \frac{M_y}{\sqrt{k}}\,\,\text{for $k\leq m_1-1$}.
\]
Similarly, we have
\[
|\hat{x}_k|\leq \frac{M_y}{\sqrt{m-k+1}},\,\,\text{for all $k\geq m_2+1$}.
\]
Combining it with $|\hat{x}|_i\leq \frac{M_y^2}{4\lambda}$ for $m_1\leq i\leq m_2$, \eqref{eq:bound} is proved.

Case 2: $\mathcal{I}$ is empty. That is, for all $0\leq i\leq m+1$, we have $|\hat{x}_i|>\frac{M_y^2}{4\lambda}$.

Case 2a: Assuming that $\sign(\hat{x}_0)\neq \sign(\hat{x}_{m+1})$.

WLOG we may assume that $\hat{x}_0$ is positive while $\hat{x}_{m+1}$ is negative. Then Lemma~\ref{lemma:localmax} and $\mathcal{I}=\emptyset$ imply that there exists $0\leq m_0\leq m+1$, such that for all $i\leq m_0$, $\hat{x}_i\geq \frac{M_y^2}{4\lambda}$, for all $i\geq m_0+1$, $\hat{x}_i\leq -\frac{M_y^2}{4\lambda}$.

Following the same proof as case 1, we have
\[
|\hat{x}_k|\leq \begin{cases} \frac{M_y}{\sqrt{k}},\,\,\text{for $1\leq k\leq m_0$}\\\frac{M_y}{\sqrt{m-k+1}},\,\,\text{for $m_0+1\leq k\leq m$},\end{cases}
\]
and \eqref{eq:bound}  is proved.

Case 2b: Assuming that $\sign(\hat{x}_0) = \sign(\hat{x}_{m+1})$.

WLOG we may assume that both  $\hat{x}_0$ and $\hat{x}_{m+1}$ are positive. Then Lemma~\ref{lemma:localmax} and $\mathcal{I}=\emptyset$ imply that $\hat{x}_i> \frac{M_y^2}{4\lambda}$  for all $0\leq i\leq m+1$, and in addition, it is first nonincreasing and then nondecreasing. That is, there exists $0\leq m_0\leq m$ such that the sequence $\hat{x}_0,\cdots,\hat{x}_{m_0}$ is nonincreasing and $\hat{x}_{m_0},\cdots,\hat{x}_{m+1}$ is nondecreasing. Therefore, if we let the set
\[
\tilde{\calI}=\{0\leq i\leq m+1: \hat{x}_i=\hat{x}_{m_0}\},
\]
then this set is an interval, that is, WLOG we can write it by $\tilde{\calI}=[m_1,m_1+1,\cdots,m_2]$.

Following the proof of case 1, we have
\begin{equation}\label{eq:est1}
|\hat{x}_k|\leq \begin{cases} \frac{M_y}{\sqrt{k}},\,\,\text{for $1\leq k\leq m_1-1$}\\\frac{M_y}{\sqrt{m-k+1}},\,\,\text{for $m_2+1\leq k\leq m$}.\end{cases}
\end{equation}
Now let us consider another sequence $\tilde{x}_i$ defined by
\[
\tilde{x}_i=\begin{cases}\hat{x}_i-\epsilon, \text{if $i\in\tilde{\mathcal{I}}$}\\\hat{x}_i, \text{if $i\not\in\tilde{\mathcal{I}}$},\end{cases}
\]
then $G(\{\hat{x}_i\}_{i=1}^m)\leq G(\{\tilde{x}_i\}_{i=1}^m)$ implies that
\[
\sum_{i\in\mathcal{I}_2}(y_i-\hat{x}_{m_0})+\lambda=0,
\]
and
\begin{equation}\label{eq:est2}
\hat{x}_{m_0}=\frac{\sum_{i\in\mathcal{I}}y_i+\lambda}{|\mathcal{I}|}.
\end{equation}
Combining it with \eqref{eq:est1}, we can analyze $\hat{x}_{m_0}$ by three cases:
\begin{itemize}
\item $m_1-1\geq m/4$. Then \eqref{eq:est1} implies that $\hat{x}_{m_0}<\hat{x}_{k_1-1}\leq \frac{2M_y}{\sqrt{m}}$.
\item $m-m_2\geq m/4$. Then \eqref{eq:est1} implies that $\hat{x}_{m_0}<\hat{x}_{k_2+1}\leq \frac{2M_y}{\sqrt{m}}$.
\item $m_1-1\leq m/4$ and $m-m_2\leq m/4$. Then we have $|\tilde{I}|\geq m/2$ and \eqref{eq:est2} implies that $\hat{x}_{m_0}\leq  \frac{2\lambda}{m}+\frac{M_y}{\sqrt{m/2}}$.
\end{itemize}
In summary, we have $\hat{x}_{m_0}\leq \frac{2\lambda}{m}+\frac{2M_y}{\sqrt{m}}$. That is, $\hat{x}_{i}\leq \frac{2\lambda}{m}+\frac{2M_y}{\sqrt{m}}$ for all $m_1\leq i\leq m_2$. Combining it with \eqref{eq:est1}, \eqref{eq:bound} is proved.

\subsection{Discussion}\label{sec:dicussion}
While we have assumed that the noise $\epsilon_i$ is i.i.d. from $N(0,\sigma^2)$, from the proof we can see that the result can be easily adapted to more generic distribution $\epsilon_i\sim \mu$. Since the only probabilistic argument in this proof is Lemma~\ref{lemma:prob},  $M_y$ can be adjusted accordingly for any new distribution $\mu$,  so that  \eqref{eq:prob_requirement} is satisfied. For example,  $M_y$ can be chosen to be $O(\sqrt{\log n})$ if $\mu$ is sub-Gaussian and  $O({\log n})$ if $\mu$ is sub-exponential.

There are some additional interesting properties of $\hat{x}_i$ arising from the proof of Theorem~\ref{thm:main}. For example, if $\sign(a)\neq\sign(b)$ then case 2b would not apply, and \eqref{eq:bound} in Theorem~\ref{thm:main} can be simplified to
\[
|\hat{x}_i|\leq \max\left(\frac{M_y}{\sqrt{i}},\frac{M_y}{\sqrt{m+1-i}},\frac{M_y^2}{4\lambda}\right).
\]
Based on it, we would have a slightly stronger result on the estimation error of fused lasso, if some additional assumptions are made, for example, if $x_{n_{k}}-x_{n_{k}-1}$ for all $k=2,\cdots,K$ always have the same signs.

Another interesting property is that, as shown in the proof of Theorem~\ref{thm:main}, when $i$ is close to the change point such that $\frac{M_y}{\sqrt{i}}$ or $\frac{M_y}{\sqrt{m+1-i}}$ dominates the bound \eqref{eq:bound}, the sequence $\hat{x}_i$ would be monotone. This property suggests that at these locations, the solution of fused lasso is the same as the solution of the isotonic regression problem. Note that the component $\sigma^2\frac{\log n}{n}\sum_{i=1}^K\log m_i$  in the estimation error bound \eqref{eq:ouroptimal} corresponds to these ``isotonic regression regions'',
applying existing results on the expected error bound of the isotonic regression estimator~\cite{zhang2002,chatterjee2015}, we can remove the factor $\log n$ from the component $\sigma^2\frac{\log n}{n}\sum_{i=1}^K\log m_i$ for the expected error bound of the fused lasso.
%

\section{Group fused lasso}\label{sec:gfl}
We first present our main result for the model \eqref{eq:gfl_model} and the group fused lasso problem \eqref{eq:gfl}, which is a parallel result to Theorem~\ref{thm:fused}:
\begin{thm}\label{thm:gfl}
Assuming $M_y=\sigma\max(\sqrt{8(\log n+\log t)+p},2\sqrt{p})$ and $25^2M_y\leq \lambda<\min_{k=1,\cdots,K}(7m_k-\sqrt{m_k})M_y$, then with probability at least $1-1/t^2$, we have
\begin{align}\label{eq:main1}
\|\hat{\bx}_i-\bx_i\|\leq &\max\left(\frac{4(M_y\sqrt{m_{k(i)}/2}+\lambda)}{m_{k(i)}}, \frac{50M_y}{\sqrt{m_{k(i)}}}, \frac{25\sqrt{5}M_y}{\sqrt{d_i}}, 125\sqrt{\frac{M_y^3}{\lambda}}\right)
\end{align}
\end{thm}
Similar to Section~\ref{sec:fused2}, we can derive the properties of the group fused lasso solution in both estimation error and approximate change point detection. Their proofs are skipped since they are straightforward generalizations of the proofs in Section~\ref{sec:fused2}.
\begin{cor}\label{cor:estimationerror1} With the assumptions in Theorem~\ref{sec:gfl}, we have the upper bound of the estimation error for the group fused lasso problem \eqref{eq:gfl} with probability $1-1/t^2$:
\[
\sum_{i=1}^n\|\hat{\bx}_i-\bx_i\|^2\leq 6250M_y^2\sum_{i=1}^K\log m_i +6000K M_y^2 +\frac{32\lambda^2K}{m_H}+\frac{125^2M_y^6n}{\lambda}.
\]
\end{cor}

\begin{cor}\label{cor:detection1}
Assume $H_n$ is the smallest distances between the ``levels'' defined by\\ $H_n=\min_{k=1,\cdots,K-1}\|\bx_{n_k}-\bx_{n_{k+1}}\|$ and recall that $W_n=\min_{1\leq k\leq K}m_k$.  When $H_n=96CM_y/W_n^{2/3}$ with $C>1$, the following procedure can be used to detect the change points: $\lambda=(6C-2)M_yW_n^{2/3}$, then for any $n_i+\frac{5}{24C}W_n^{1/3}\leq i\leq n_{i+1}-\frac{5}{24C}W_n^{1/3}$, we have that with probability at least $1-1/t^2$,
\[
\|\hat{\bx}_i-\bx_i\|\leq H_n/4.
\]
As a result, the following detection scheme
\[
\hat{S}=\{1\leq i\leq n: \|\hat{\bx}_{i-\frac{5}{24C}W_n^{1/3}}-\hat{\bx}_{i+\frac{5}{24C}W_n^{1/3}}\|<H_n/2\}
\]
would have the detection error of
\[
d_H(\hat{S},S)\leq \frac{5}{12C}W_n^{1/3}.
\]
\end{cor}
The proof of Corollary~\ref{cor:detection1} is similar to the proof of Corollary~\ref{cor:detection}, by noting that $\frac{4(M_y\sqrt{m_{k(i)}/2}+\lambda)}{m_{k(i)}}$ is the dominant term in the RHS of \eqref{eq:main1}.

We remark that the result here is a little bit worse than the corresponding result in Section~\ref{sec:fused}. In particular, compare the upper bounds in Theorem~\ref{thm:gfl} and  Theorem~\ref{thm:fused}, the main difference is that \eqref{eq:main1} has a term in the order of $O(\sqrt{1/\lambda})$ while \eqref{eq:main01} has a corresponding term in the order of $O({1/\lambda})$. This means that the results in Theorem~\ref{thm:gfl} and is a little bit worse than the result for fused lasso in  Section~\ref{sec:fused}, and the main reason is that our analysis is based on the ``direction'' of $\hat{x}_i-\hat{x}_{i+1}$ for $1\leq i\leq n$. If $x_i$ is a scalar, then this direction is either $1$, $0$, or $-1$. But if $\bx_i$ is a vector, then this direction could be any unit vector and have infinite possibilities. This makes the group fused lasso more complicated and difficult to analyze.  Based on Theorem~\ref{thm:gfl}, Corollaries~\ref{cor:estimationerror1} and~\ref{cor:detection1} are also slightly weaker than Corollaries~\ref{cor:estimationerror} and~\ref{cor:detection} respectively. For example, Corollary~\ref{cor:detection1} requires $H_nW_n^{2/3}>O(\sqrt{\log n})$ while Corollary~\ref{cor:detection} requires $H_nW_n^{1/2}>O(\sqrt{\log n})$.

However, we expect that this is just an artifact due to the technique of our proof in this section, and simulations seem to suggest the same order of $1/\lambda$ as in \eqref{eq:bound}. We leave this possible improvement to future works.

\subsection{Proof of Theorem~\ref{thm:gfl}}
Similar to the proof of Theorem~\ref{thm:fused}, its proof is based on the property of the solution to
\begin{align}\label{eq:optimization1_group}
&\{\hat{\bx}_i\}_{i=1}^m=\argmin_{\{\bx_i\}_{i=1}^m\subset\reals^p}G(\bx_1,\cdots,\bx_m),\,\,\text{for}\\
&G(\bx_1,\cdots,\bx_m)=\sum_{i=1}^m\|\bx_i-\by_i\|^2+\lambda\left(\|\ba-\bx_1\|+\|\bx_m-\bb\|+\sum_{i=1}^{m-1}\|\bx_i-\bx_{i+1}\|\right).
\end{align}
\begin{thm}\label{thm:main1}
If
\begin{equation}\label{eq:main_assumption}
\left\|\sum_{i=k}^l \by_i\right\|\leq M_y \sqrt{l-k+1}
\end{equation}
 for all $1\leq k\leq l\leq m$, and $25^2M_y\leq \lambda<(7m-\sqrt{m})M_y$, then the solution defined in \eqref{eq:optimization1_group} satisfies
\begin{align}\label{eq:bound1}
\|\hat{\bx}_i\|\leq &\max\left(\frac{4(M_y\sqrt{m/2}+\lambda)}{m}, \frac{50M_y}{\sqrt{m}}, \frac{25\sqrt{5}M_y}{\sqrt{i}}, \frac{25\sqrt{5}M_y}{\sqrt{m-i}}, 125\sqrt{\frac{M_y^3}{\lambda}}\right).
\end{align}
\end{thm}
Similar to Lemma~\ref{lemma:prob}, we first present a probabilistic result and then prove Theorem~\ref{thm:gfl} based on the assumption that this probabilistic event holds:
\begin{lemma}\label{lemma:prob1}
With probability $1-1/t^2$, for all $1\leq k\leq l\leq n$, we have
\[
\|\sum_{i=k}^l \bepsilon_i\|\leq M_y \sqrt{l-k+1}
\]
with $M_y=\sigma\max(\sqrt{8(\log n+\log t)+p},2\sqrt{p})$.
\end{lemma}
\begin{proof}
Following \cite[Proposition 2.2]{measure2005}, for any $\delta>3$ and $\by\sim N(0,\bI)\in\reals^p$, we have
\[
\Pr(\|\by\|^2\geq (1+\delta)p)\leq \exp(-\frac{p}{2}(\delta-\ln(1+\delta)))\leq \exp(-\frac{p}{4}\delta)
\]
So for any fixed $1\leq k\leq l\leq n$ and $M_y\geq 2\sigma\sqrt{p}$, we have
\[
\|\sum_{i=k}^l \bepsilon_i\|\leq M_y \sqrt{p(l-k+1)}
\]
holds with probability $
1-\exp\left(-\frac{p}{4}(M_y^2-1)\right)$.
Applying a union bound, it holds for all $1\leq k\leq l\leq n$ with probability at least
\[
1-n^2\exp\left(-\frac{p}{4}(M_y^2-1)\right)\geq 1-1/t^2.
\]
\end{proof}
Similar to the proof of Theorem~\ref{thm:fused}, the combination of Theorem~\ref{thm:main1} and Lemma~\ref{lemma:prob1} implies Theorem~\ref{thm:gfl}.
\subsection{Proof of Theorem~\ref{thm:main1}}
For simplicity and consistency of the notation, we denote $\ba$ and $\bb$ by $\hat{\bx}_0$ and $\hat{\bx}_{m+1}$ respectively. For the proof, we use the notations as follows: for any $1\leq i\leq m$, let $p(i)$ and $n(i)$ be the integers between $0$ and $m+1$ such that
\[
\hat{\bx}_{p(i)-1}\neq \hat{\bx}_{p(i)}=\cdots= \hat{\bx}_{i}\cdots= \hat{\bx}_{n(i)}\neq \hat{\bx}_{n(i)+1}.
\]
We let $p(i)=0$ if $\hat{\bx}_{0}=\hat{\bx}_{1}=\cdots=\hat{\bx}_{i}$, and similarly $n(i)=m+1$ of $\hat{\bx}_i=\cdots=\hat{\bx}_{m+1}$.  As a special example, if $\hat{\bx}_{i-1}\neq \hat{\bx}_{i}\neq \hat{\bx}_{i+1}$, then $p(i)=i=n(i)$.

We let $\bu_{i}=\frac{\hat{\bx}_{i+1}-\hat{\bx}_{i}}{\|\hat{\bx}_{i+1}-\hat{\bx}_{i}\|}$ if $\hat{\bx}_{i}\neq \hat{\bx}_{i+1}$; and $\bu_{i}=\b0$ (or undefined?) otherwise.

Then we have the following lemma:
\begin{lemma}\label{lemma:deri}
For $k\leq l$, if $\hat{\bx}_{k-1}\neq \hat{\bx}_{k}$ and $\hat{\bx}_{l}\neq \hat{\bx}_{l+1}$, then
\[
2\sum_{i=k}^{l}(\hat{\bx}_i-\by_i)+\lambda(\bu_{k-1}-\bu_{l})=\b0.
\]
\end{lemma}
\begin{proof}
The proof is based on the fact that for
\[
f(\bz)=G(\hat{\bx}_1,\cdot,\hat{\bx}_{k-1},\hat{\bx}_{k}+\bz,\cdots,\hat{\bx}_{l}+\bz,\hat{\bx}_{l+1},\cdots,\hat{\bx}_{m}),
\]
$f(\bz)$ is minimized at $\bz=\b0$. Then this lemma follows from  $f'(\b0)=0$.
\end{proof}
Applying that $\|\bu_{k-1}\|=\|\bu_{l}\|=1$, this lemma has two consequences that will be used later: If we use $\angle(\bu,\bv)=\cos^{-1}\frac{\bu^T\bv}{\|\bu\|\|\bv\|}$ to represent the angle between vectors $\bu$ and $\bv$, then

1. $\angle(\bu_{l}, \sum_{i=k}^{l}(\hat{\bx}_i-\by_i))>\pi/2$.

2. $\angle(\bu_{l},\bu_{k-1})=2\angle(\bu_{l}, \sum_{i=k}^{l}(\hat{\bx}_i-\by_i))-\pi$.

The next lemma controls the largest $\|\hat{\bx}_i\|$ through $M_y$.
\begin{lemma}\label{lemma:upperbound}
Under the event of Lemma~\ref{lemma:prob1} and assuming that $\lambda<(7m-\sqrt{m})M_y$, then we have that $\max_{1\leq i\leq m}\|\hat{\bx}_i\|\leq M_d$, where $M_d=25M_y$.
\end{lemma}
\begin{proof}
Let $M_d'=\max_{1\leq i\leq m}\|\bx_i\|$ and $k=\argmax_{1\leq i\leq m}\|\bx_i\|$, and we will prove Lemma~\ref{lemma:upperbound} by contradiction: we assume that $M_d'>M_d$ and look for contradictions by investigating difference cases:

Case 1.  $p(k)>1$ and $n(k)<m$. Then by assumption, we have $\|\hat{\bx}_{k}\|=M_d'$, $\|\hat{\bx}_{p(k)-1}\|, \|\hat{\bx}_{n(k)+1}\|\leq M_d'$, and $\hat{\bx}_{p(k)-1}, \hat{\bx}_{n(k)+1}\neq \hat{\bx}_k$.

Then we have $\bu_{p(k)-1}^T\hat{\bx}_k>0$, $\bu_{n(k)}^T\hat{\bx}_k<0$. In addition, \[
\hat{\bx}_k^T\sum_{i=p(k)}^{n(k)}(\hat{\bx}_k-\by_i)\geq (n(k)-p(k)+1)M_d'^2-\sqrt{n(k)-p(k)+1}M_d'M_y> 0.\]
Combining the three inequalities above, we have a contradiction to Lemma~\ref{lemma:deri} since $M_d'>M_y.$

Case 2. $p(k)=0$ or $1$. For simplicity, WLOG we may assume that $k=n(k)$, that is, $\hat{\bx}_1=\cdots =\hat{\bx}_k\neq \hat{\bx}_{k+1}$ (if $k=m+1$, then the proof would follow the proof in case 2b-I after \eqref{eq:dist}). By assumption, we have $\|\hat{\bx}_{k+1}\|\leq M_d'$, and as a result,
\[
\bu_k^T\frac{\hat{\bx}_k}{\|\hat{\bx}_k\|}=\frac{\hat{\bx}_{k+1}-\hat{\bx}_k}{\|\hat{\bx}_{k+1}-\hat{\bx}_k\|}^T\frac{\hat{\bx}_k}{\|\hat{\bx}_k\|}\leq 0.\]

Now we consider two cases:

Case 2a. $\bu_k^T\frac{\hat{\bx}_k}{\|\hat{\bx}_k\|}<-c_0$, where $c_0=0.05$.

For this case, consider $\tilde{\bx}_i=\hat{\bx}_i+\epsilon\bu_k$ for $1\leq i\leq k$ and $\tilde{\bx}_i=\hat{\bx}_i$ otherwise. Then we have
\begin{equation}
\|\hat{\bx}_i-\hat{\bx}_{i+1}\|-\|\tilde{\bx}_i-\tilde{\bx}_{i+1}\|=\begin{cases}\epsilon\|\bu_k\|,\,\,&\text{if $i=k$}\\
\|\hat{\bx}_0-\hat{\bx}_{1}\|-\|\hat{\bx}_0-\hat{\bx}_{1}-\epsilon\bu_k\|\geq -\epsilon\|\bu_k\|,\,\,\,&\text{if $i=0$}\\
0,\,\,&\text{if $i\neq 0,k$}.\end{cases}
\end{equation}
Therefore,\[
\sum_{i=0}^{m}\|\hat{\bx}_i-\hat{\bx}_{i+1}\|\geq \sum_{i=0}^{m}\|\tilde{\bx}_i-\tilde{\bx}_{i+1}\|.
\]
Combining it with
\[
\sum_{i=1}^m\|\hat{\bx}_i-\by_i\|^2-\sum_{i=1}^m\|\tilde{\bx}_i-\by_i\|^2+O(\epsilon^2)
=-2\epsilon\bu_k^T\sum_{i=1}^k(\hat{\bx}_i-\by_i)\geq 2\epsilon(c_0kM_d'-M_y\sqrt{k})+O(\epsilon^2)
\]
and \begin{equation}c_0>\frac{M_y}{M_d'},\label{eq:require1}\end{equation} we have a contradiction to $G(\{\hat{\bx}_i\}_{i=1}^m)\leq G(\{\tilde{\bx}_i\}_{i=1}^m)$.

Case 2b. $-c_0\leq \bu_k^T\frac{\hat{\bx}_k}{\|\hat{\bx}_k\|}\leq 0$. Then we let $l=\arg\min_{j=k+1, \cdots, m}\sum_{i=k}^{j-1}\|\hat{\bx}_i-\hat{\bx}_{i+1}\|\geq c_1M_d'$, where $c_1=0.1$.

Case 2b-I: If such $l$ does not exist, then it means that \begin{equation}\label{eq:dist}\text{$\|\hat{\bx}_i-\hat{\bx}_k\|\leq c_1\|\hat{\bx}_k\|$ for all $1\leq i\leq m$.}\end{equation} Consider the sequence $\tilde{\bx}_i=\hat{\bx}_i-\epsilon\hat{\bx}_k$, then we have
\[
\sum_{i=0}^{m}\|\hat{\bx}_i-\hat{\bx}_{i+1}\|-\sum_{i=0}^{m}\|\tilde{\bx}_i-\tilde{\bx}_{i+1}\|\geq -2\epsilon M_d'
\]
and
\[
\sum_{i=1}^m\|\hat{\bx}_i-\by_i\|^2-\sum_{i=1}^m\|\tilde{\bx}_i-\by_i\|^2
=2\epsilon\hat{\bx}_k^T\sum_{i=1}^m(\hat{\bx}_i-\by_i)\geq 2\epsilon\left(m(1-c_1) M_d'^2-M_d' M_y\sqrt{m}\right).
\]
Combining these two estimations with \begin{equation}m (1-c_1) M_d' - M_y\sqrt{m}-\lambda>0,\label{eq:require2}\end{equation}  we have a contradiction to $G(\{\hat{\bx}_i\}_{i=1}^m)\leq G(\{\tilde{\bx}_i\}_{i=1}^m)$.

Case 2b-II: Now assuming that $l\leq m$. Then  for all $k+1\leq j\leq l-1$, $\|\hat{\bx}_j-\hat{\bx}_k\|\leq c_1\|\hat{\bx}_k\|$ and
\[
\left\|\sum_{i=k+1}^{j}(\hat{\bx}_i-\by_i)-(j-k)\hat{\bx}_k\right\|\leq c_1(j-k)\hat{\bx}_k+M_y\sqrt{j-k},
\]
which implies that $\angle(\sum_{i=k+1}^{j}(\hat{\bx}_i-\by_i),\hat{\bx}_k)\leq \sin^{-1}\left(c_1+\frac{M_y}{M_d'}\right)$.
Combining it with $\angle(\bu_k,\hat{\bx}_k)=\sin^{-1}c_0+\pi/2$ and the analysis after Lemma~\ref{lemma:deri},
\[
\angle(\bu_k,\bu_j)\leq 2\sin^{-1}c_0+2\sin^{-1}\left(c_1+\frac{M_y}{M_d'}\right).
\]
Therefore, for any $k\leq j_1\leq j_2\leq l-1$,
\begin{equation}\label{eq:bu2}
\angle(\bu_{j_1},\bu_{j_2})\leq 4\sin^{-1}c_0+4\sin^{-1}\left(c_1+\frac{M_y}{M_d'}\right).
\end{equation}
Let $c_2=\cos\left(4\sin^{-1}c_0+4\sin^{-1}\left(c_1+\frac{M_y}{M_d'}\right)\right)$, then \begin{equation}\label{eq:bu3}
\text{$c_2>0.7$ and $\bu_{j_1}^T\bu_{j_2}>c_2$.}
\end{equation}
Let $\hat{\bx}_l'$ be the point on the line segment connecting $\hat{\bx}_{l-1}$ and $\hat{\bx}_{l}$ such that
\begin{equation}\label{eq:bu4}
\sum_{i=k}^{l-1}\|\hat{\bx}_i-\hat{\bx}_{i+1}\|+\|\hat{\bx}_{l-1}-\hat{\bx}_{l}'\|=c_1M_d',
\end{equation}
\begin{align}
&\bu_{l-1}^T\hat{\bx}_l'-\bu_{k}^T\hat{\bx}_k\label{eq:endofline}
=\bu_{l-1}^T(\hat{\bx}_l'-\hat{\bx}_k)+(\bu_{l-1}-\bu_{k})^T\hat{\bx}_k\\
=&\bu_{l-1}^T\left(\sum_{i=k}^{l-2}\|\hat{\bx}_{i+1}-\hat{\bx}_i\|\bu_i+\|\hat{\bx}_{l}'-\hat{\bx}_{l-1}\|\bu_{l-1}\right)+\frac{1}{\lambda}\left(\sum_{i=k+1}^{l-1}(\hat{\bx}_i-\by_i)\right)^T\hat{\bx}_k\nonumber\\\nonumber
\geq& c_2 c_1M_d'+\frac{1}{\lambda}\left((l-k-1)(1-c_1)M_d'^2-\sqrt{l-k-1}M_d'M_y\right),
\end{align}
where the inequality applies \eqref{eq:bu3} and \eqref{eq:bu4}. Combining it with $\bu_{k}^T\hat{\bx}_k\geq -c_0M_d'$ and \begin{equation}
(1-c_1)M_d'>M_y,\,\,c_2 c_1>c_0.
\label{eq:require3}\end{equation}
we have $\bu_{l-1}^T\hat{\bx}_l'\geq 0$. In addition, \eqref{eq:bu4} implies $\|\hat{\bx}_l'\|\geq M_d'(1-c_1)$.

Let $\bd=\hat{\bx}_l'/\|\hat{\bx}_l'\|$ be the unit vector in the direction of $\hat{\bx}_l'$, we will use induction to prove that
\begin{align}\label{eq:final_argument}
\text{$\bd^T\bu_{i}>0$ and $\bd^T\hat{\bx}_i$ is nondecreasing for all for all $i\geq l-1$.}
\end{align}
If \eqref{eq:final_argument} holds, then we have $\|\bd^T\hat{\bx}_i\|\geq \|\bd^T\hat{\bx}_l\|\geq \|\bd^T\hat{\bx}_l'\|=(1-c_1)M_d'$ for all $i\geq l$. As a result,
\[
\|\hat{\bx}_i-\hat{\bx}_l'\|^2\leq 2M_d'^2-2\hat{\bx}_l'^T\hat{\bx}_i\leq 2M_d'^2-2(1-c_1)M_d'^2=(4c_1-2c_1^2)M_d'^2.
\]
Combining it with $\|\hat{\bx}_k-\hat{\bx}_l'\|\leq c_1M_d'$ and $\hat{\bx}_1=\cdots=\hat{\bx}_k$, we have that for all $1\leq i\leq m$,
\[
\|\hat{\bx}_i-\hat{\bx}_l'\|\leq (c_1+\sqrt{4c_1-2c_1^2})M_d'.
\]
Now following the analysis for the case 2b-I after \eqref{eq:dist} (with $c_1=0.1$ replaced by $c_1+\sqrt{4c_1-2c_1^2}<0.72$), we have the contradiction since
\begin{equation}m (1-c_1-\sqrt{4c_1-2c_1^2}) M_d - M_y\sqrt{m}-\lambda>0.\label{eq:require4}\end{equation}

It remains to prove \eqref{eq:final_argument}, which can be proved using induction. Considering that for all $l\leq i\leq n(l)$,
\[
\bd^T(\hat{\bx}_i-\hat{\by}_i)\geq \bd^T(\hat{\bx}_l'-\hat{\by}_i)\geq (1-c_1)M_d'-M_y>0,
\]
so Lemma~\ref{lemma:deri} (note that here $p(l)=l$) implies that
\[
\bd^T\bu_{n(l)}\geq \bd^T\bu_{l-1}\geq 0.
\]
Then $\bx_{n(l)+1}=\bx_{n(l)}+\|\bx_{n(l)+1}-\bx_{n(l)}\|\bu_{n(l)}$ implies that $\bd^T\bx_{n(l)+1}\geq \bd^T\bx_{l}$. Repeat the steps above, \eqref{eq:final_argument} is proved.

Case 3. $n(k)=m$ or $m+1$. The proof is the same as the proof of case 2.
%
\end{proof}

Now we have proved that $\max_{i=1,\cdots,m}\|\hat{\bx}_i\|\leq M_d=25M_y$. Based on this result, we will prove the following lemma:
\begin{lemma}\label{lemma:order}
If $\hat{\bx}_l\neq \hat{\bx}_{l+1}$ and $\bu_{l}^T\hat{\bx}_l\leq-\sqrt{\frac{M_dM_y^2}{\lambda}}$, then  $l \leq l_0+1$ for  $l_0=\lfloor \frac{M_y^2}{|\bu_{l}^T\hat{\bx}_l|^2}\rfloor$.

Similarly, if $\hat{\bx}_l\neq \hat{\bx}_{l-1}$ and $\bu_{l-1}^T\hat{\bx}_l> \sqrt{\frac{M_dM_y^2}{\lambda}}$, then  $m-l \leq l_1+1$ for $l_1=\lfloor \frac{M_y^2}{|\bu_{l-1}^T\hat{\bx}_l|^2}\rfloor$.
\end{lemma}
\begin{proof}
Lemma~\ref{lemma:deri} implies that for $l-\frac{\lambda}{M_d}\leq i \leq l-1$, $\|\bu_i-\bu_{l}\|\leq \frac{1}{2}$, which implies that $\bu_{l}^T\bu_i>0$. In addition, we have
\[\hat{\bx}_l-\hat{\bx}_k=\sum_{i=k}^{l-1}\bu_i\|\hat{\bx}_i-\hat{\bx}_{i+1}\|\] by definition. Combining it with $l_0\leq \frac{\lambda}{4M_d}$, for any $l-l_0 \leq k \leq l-1$,
\begin{equation}\label{eq:temp}
\bu_{l}^T(\hat{\bx}_l-\hat{\bx}_k)=\sum_{i=k}^{l-1}\bu_{l}^T\bu_i\|\hat{\bx}_i-\hat{\bx}_{i+1}\|\geq 0.
\end{equation}
By contradiction we assume that $l>l_0+1$, and consider the $\tilde{\bx}_i=\hat{\bx}_i+\epsilon\bu_{l}$ for all $l-l_0\leq i\leq l$ and $\tilde{\bx}_i=\hat{\bx}_i$ otherwise, then we have
\begin{align}\label{eq:temp1}
\sum_{i=0}^{m}\|\hat{\bx}_i-\hat{\bx}_{i+1}\|-\sum_{i=0}^{m}\|\tilde{\bx}_i-\tilde{\bx}_{i+1}\|\geq 0
\end{align}
and
\begin{align}\label{eq:temp2}
&\sum_{i=1}^m\|\hat{\bx}_i-\by_i\|^2-\sum_{i=1}^m\|\tilde{\bx}_i-\by_i\|^2
=-2\epsilon\bu_{l}^T\sum_{i=l-l_0}^{l}(\hat{\bx}_i-\by_i)\geq-2\epsilon\bu_{l}^T\sum_{i=l-l_0}^{l}(\hat{\bx}_l-\by_i)\\\geq &2\epsilon((l_0+1)|\bu_{l}^T\hat{\bx}_l|-M_y\sqrt{l_0+1})> 0.\nonumber
\end{align}
where the first inequality applies \eqref{eq:temp} and the last inequality follows from the definition of $l_0$. Combining  \eqref{eq:temp1} and \eqref{eq:temp2}, we have a contradiction to $G(\{\tilde{\bx}_i\}_{i=1}^m)\geq G(\{\hat{\bx}_i\}_{i=1}^m)$.
\end{proof}

\begin{lemma}\label{lemma:main}
Assuming $\lambda\geq M_d^2/M_y$ and $c_1=1/25$, then if $\|\hat{\bx}_l\|>\frac{1}{c_1}\sqrt{\frac{M_dM_y^2}{\lambda}}$ and $\frac{5}{c_1^2}\frac{M_y^2}{\|\hat{\bx}_l\|^2} \leq l\leq m-\frac{5}{c_1^2}\frac{M_y^2}{\|\hat{\bx}_l\|^2}$, then for all  $\frac{1}{c_1^2}\frac{M_y^2}{\|\hat{\bx}_l\|^2}+1\leq i\leq m-\frac{1}{c_1^2}\frac{M_y^2}{\|\hat{\bx}_l\|^2}-1$, $\|\hat{\bx}_i-\hat{\bx}_l\|\leq \|\hat{\bx}_l\|/2$.
\end{lemma}
\begin{proof}
We will prove it  by dividing it into the following cases.

Case 1: $\bu_{n(l)}^T\hat{\bx}_{n(l)}/\|\hat{\bx}_{n(l)}\|\leq -c_1$.
By Lemma~\ref{lemma:order}, then $\bu_{n(l)}^T\hat{\bx}_{n(l)}\leq -c_1\|\hat{\bx}_l\|\leq-\sqrt{\frac{4M_dM_y^2}{\lambda}}$ and $l\leq n(l)\leq \frac{M_y^2}{\|\hat{\bx}_l\|^2c_1^2}+1.$

Case 2:  $\bu_{n(l)}^T\hat{\bx}_{n(l)}/\|\hat{\bx}_{n(l)}\|\geq 2 c_1$.
Then we have $\hat{\bx}_{n(l)}\neq \hat{\bx}_{n(l)+1}$ and $\bu_{n(l)}^T\hat{\bx}_{n(l)+1} > \bu_{n(l)}^T\hat{\bx}_{n(l)}\geq  2c_1\|\hat{\bx}_{n(l)}\|$, and Lemma~\ref{lemma:order} implies that \begin{equation}n(l)+1\geq m-\frac{M_y^2}{4\|\hat{\bx}_{n(l)}\|^2c_1^2}.\label{eq:case2}\end{equation} This shows that $n(l)-p(l)\geq n(l)-l\geq \frac{19M_y^2}{4\|\hat{\bx}_{n(l)}\|^2c_1^2}$.

Now let us consider $\hat{\bx}_{p(l)}^T\bu_{p(l)-1}$. Applying  \[
\angle(\hat{\bx}_l,\sum_{i=p(l)}^{n(l)} (\hat{\bx}_i-\by_i))\leq \sin^{-1}\left(\frac{M_y}{\|\hat{\bx}_l\|\sqrt{n(l)-p(l)}}\right)
\]
with the implication of Lemma~\ref{lemma:deri} that
\[
\angle(-\bu_{p(l)-1},\sum_{i=p(l)}^{n(l)} (\hat{\bx}_i-\by_i))=\angle(\bu_{n(l)},\sum_{i=p(l)}^{n(l)} (\hat{\bx}_i-\by_i)),
\]
we have
\[
\angle(-\bu_{p(l)-1}, \hat{\bx}_l)\geq \sin^{-1}(2c_1)-2\sin^{-1}\left(\frac{M_y}{\|\hat{\bx}_l\|\sqrt{n(l)-p(l)}}\right)\geq \sin^{-1}(2c_1)-2\sin^{-1}\left({\frac{2c_1}{\sqrt{19}}}\right).
\]
Since (for all $c_1\leq 0.5$),
\[
\sin^{-1}(2c_1)-2\sin^{-1}\left({\frac{2c_1}{\sqrt{19}}}\right)>\sin^{-1}(c_1),
\]
 we have $\bu_{p(l)-1}^T\hat{\bx}_{p(l)}/\|\hat{\bx}_{p(l)}\|\leq -c_1$. Combining it with  $\bu_{p(l)-1}^T\hat{\bx}_{p(l)-1}\leq \bu_{p(l)-1}^T\hat{\bx}_{p(l)}$ and Lemma~\ref{lemma:order},  we have  $p(l)\leq \frac{M_y^2}{\|\hat{\bx}_l\|^2c_1^2}+1$. Combining it with \eqref{eq:case2}, the lemma is proved.

Case 3: $\bu_{p(l)-1}^T\hat{\bx}_l/\|\hat{\bx}_l\|\geq c_1$, or $\leq -2c_1$. Then the same analysis from cases 1 and 2 can be applied.

Case 4:  $-c_1\leq \bu_{n(l)}^T\hat{\bx}_l/\|\hat{\bx}_l\|\leq 2c_1$ and $-2c_1\leq \bu_{p(l)-1}^T\hat{\bx}_l/\|\hat{\bx}_l\|\leq c_1$.

Let $k=\min\{j\geq l+1: \sum_{i=l}^{j-1}\|\hat{\bx}_i-\hat{\bx}_{i+1}\|>c_3\|\hat{\bx}_l\|\}$ with $c_3=19c_1/2=19/50$. Then for all $n(l)+\frac{\lambda}{4M_d}\leq j\leq k-1$, we have
\[
\angle\left(\hat{\bx}_l,\sum_{i=n(l)+1}^{j} (\hat{\bx}_i-\by_i)\right)\leq \sin^{-1}\left(\frac{M_y}{\|\hat{\bx}_l\|\sqrt{j-n(l)}}+c_3\right)\leq \sin^{-1}\left(\frac{2M_y\sqrt{M_d}}{\|\hat{\bx}_l\|\sqrt{\lambda}}+c_3\right)
\leq \sin^{-1}\left(c_1+c_3\right)
\]
and Lemma~\ref{lemma:deri} implies that for all $n(l)\leq j \leq n(l)+\frac{\lambda}{4M_d}$,\[
\angle(\bu_j,\bu_{p(l)-1})\leq 2\sin^{-1}(2c_1)+2\sin^{-1}\left(c_1+c_3\right)\leq \pi/3
\]
In addition, for all $n(l)\leq j\leq n(l)+\frac{\lambda}{4M_d}$,
\[
\|\bu_j-\bu_{n(l)-1}\|\leq \frac{4M_d}{\lambda}(j-n(l)).
\]
In summary, for all $n(l)\leq j\leq k-1$, we have $\bu_j^T\bu_{n(l)-1}\geq 1/2$.

Now let us consider $\bu_{k-1}^T\hat{\bx}_k$:
\begin{align*}
&\bu_{k-1}^T\hat{\bx}_k-\bu_{n(l)}^T\hat{\bx}_{n(l)}
=\bu_{k-1}^T(\hat{\bx}_{k}-\hat{\bx}_{n(l)})+(\bu_{k-1}-\bu_{n(l)})^T\hat{\bx}_{n(l)}\\
=&\bu_{k-1}^T\left(\sum_{i=n(l)}^{k-1}\|\hat{\bx}_{i+1}-\hat{\bx}_i\|\bu_i\right)+\frac{1}{\lambda}\left(\sum_{i=n(l)+1}^{k-1}(\hat{\bx}_i-\by_i)\right)^T\hat{\bx}_{n(l)}\\
\geq& c_3\|\hat{\bx}_l\|/2+\frac{1}{\lambda}\left((1-c_3)(k-n(l)-1)\|\hat{\bx}_l\|^2-\sqrt{k-n(l)-1}M_y\|\hat{\bx}_l\|\right)\\
\geq& c_3\|\hat{\bx}_l\|/2-\frac{M_y^2}{2\lambda(1-c_3)}
\geq \left(\frac{c_3}{2}-\frac{c_1M_y}{(1-c_3)\sqrt{\lambda M_d}}\right)\|\hat{\bx}_l\|.
\end{align*}
Since
\[
\left(\frac{c_3}{2}-\frac{c_1M_y}{(1-c_3)\sqrt{\lambda M_d}}-2c_1\right)\geq c_1,
\]
Lemma~\ref{lemma:order} implies that $m-k\leq \frac{M_y^2}{c_1^2\|\hat{\bx}_l\|^2}+1$.
Similarly, if we define $k'=\max\{j\leq l-1: \sum_{i=j}^{l-1}\|\hat{\bx}_i-\hat{\bx}_{i+1}\|>c_3\|\hat{\bx}_l\|\}$, then $k'\geq  \frac{M_y^2}{c_1^2\|\hat{\bx}_l\|^2}+1$. In summary, for all $\frac{M_y^2}{c_1^2\|\hat{\bx}_l\|^2}+1<i<m-(\frac{M_y^2}{c_1^2\|\hat{\bx}_l\|^2}+1)$, we have
\[
\|\hat{\bx}_i-\hat{\bx}_l\|\leq \|\hat{\bx}_l\|/2.
\]
Combining the four cases together, Lemma~\ref{lemma:main} is proved.
\end{proof}
\begin{proof}[Proof of Theorem~\ref{thm:main1}]
Consider the sequence $\tilde{\bx}_i=\hat{\bx}_i-\epsilon\hat{\bx}_l$ for all $c_4 \leq i\leq m-c_4$ with  $c_4=\frac{M_y^2}{c_1^2\|\hat{\bx}_l\|^2}+1$, and  $\tilde{\bx}_i=\hat{\bx}_i$ otherwise, then
\begin{align}\label{eq:temp1}
\sum_{i=0}^{m}\|\hat{\bx}_i-\hat{\bx}_{i+1}\|-\sum_{i=0}^{m}\|\tilde{\bx}_i-\tilde{\bx}_{i+1}\|\geq -2\epsilon \|\hat{\bx}_l\|.
\end{align}
By Lemma~\ref{lemma:main}, if
\[
\|\hat{\bx}_i\|\geq \max\left( \frac{25\sqrt{5}M_y}{\sqrt{i}}, \frac{25\sqrt{5}M_y}{\sqrt{m-i}}, 125\sqrt{\frac{M_y^3}{\lambda}}\right),\]
then
\begin{align}\label{eq:temp2}
&\sum_{i=1}^m\|\hat{\bx}_i-\by_i\|^2-\sum_{i=1}^m\|\tilde{\bx}_i-\by_i\|^2
=-2\epsilon\hat{\bx}_{l}^T\sum_{i=c_4}^{m-c_4}(\hat{\bx}_i-\by_i)+O(\epsilon^2)\\\geq &2\epsilon\left(\frac{m-2c_4}{2}\|\hat{\bx}_l\|^2-M_y\|\hat{\bx}_l\|\sqrt{m-2c_4}\right)+O(\epsilon^2).\nonumber
\end{align}
Combining it with \eqref{eq:temp1}
 and  $G(\{\tilde{\bx}_i\}_{i=1}^m)\geq G(\{\hat{\bx}_i\}_{i=1}^m)$, we have\[
\|\hat{\bx}_l\|\leq \frac{2(M_y\sqrt{m-2c_4}+\lambda)}{m-2c_4}.
\]
If $\|\hat{\bx}_l\|\geq 50M_y/\sqrt{m}$, then $c_4\leq m/4$ and  $\|\hat{\bx}_l\|\leq \frac{2(M_y\sqrt{m-2c_4}+\lambda)}{m-2c_4}\leq \frac{4(M_y\sqrt{m/2}+\lambda)}{m}$, and   Theorem~\ref{thm:main1} is proved.
\end{proof}
.
%
%
%
\section{Conclusion}
In this work, we analyze the total variation regularized estimator (fused lasso) for the problem of change point detection and give element-wise upper bounds of the estimation errors. Based on this estimation, we show that a simple thresholding procedure can be used to detect the change points from the solution to the fused lasso. In addition, we also generalize this analysis to the setting of group fused lasso.

In the future, we plan to continue to work on some unsolved questions in this work. For fused lasso, as discussed in Section~\ref{sec:dicussion}, there exists settings where the estimation error bounds can be improved, and we hope to find a coherent way to describe such scenarios. For group fused lasso, we hope to improve the estimation error bounds from $O(1/\sqrt{\lambda})$ in~\eqref{thm:main1} to $O(1/{\lambda})$, so that it matches the estimation error bounds of the univariate setting. We also plan to investigate the possibility of generalizing this method to graph fused lasso~\cite{Hallac2015,Tansey2015,barberoTV14}, trend filtering~\cite{tibshirani2014,Ramdas2015}, or the structural change estimation problem~\cite{doi:10.1162/003465397557132,doi:10.1002/jae.659,Qian2016}.

\section{Acknowledgement}
The author would like to thank Amit Singer for introducing this interesting problem and  related literature.

\bibliographystyle{abbrv}
\bibliography{bib-online,bibliogr,bib-online2}
\end{document}